\documentclass{amsart}[12pt]
\def\doctype{}

\usepackage{latexsym,amssymb,bm}
\usepackage{color}
\usepackage{dsfont}
\usepackage{fancyhdr}
\usepackage{subfig}
\usepackage{tikz}
\usepackage{hyperref}
\hypersetup{
colorlinks=true,
citecolor=blue,
linkcolor=blue,
  urlcolor=blue
}

\newcommand\Z{\mathbb{Z}}
\newcommand\Q{\mathbb{Q}}

\newcommand\bal{\mathrm{bal}}

\newcommand{\comment}[1]{}
\newcommand{\perm}[1]{{\small \mbox{\tt{#1}}}}

\numberwithin{equation}{section}

\pgfdeclarelayer{edgelayer}
\pgfdeclarelayer{nodelayer}
\pgfsetlayers{edgelayer,nodelayer,main}
\tikzstyle{none}=[inner sep=0pt]
\tikzstyle{vertex}=[draw=black, line width=0.3mm, shape = circle, inner sep=2pt]
\tikzstyle{edge}=[draw=black, line width=0.3mm]

\fancypagestyle{titlepage}{
\fancyhead[R]{\doctype}
\fancyhead[CL]{}
\cfoot{\vspace{5pt} \thepage}
}

\let\oldsection\section
\newcommand\boldsection[1]{\oldsection{\bf #1}}
\newcommand\starsection[1]{\oldsection*{\bf #1}}
\makeatletter
\renewcommand\section{\@ifstar\starsection\boldsection}
\makeatother

\newtheoremstyle{theorem}
  {12pt}		  % space above
  {0pt}  % space below
  {\sl}  % bofy font
  {\parindent}     % ident - empty=no indent,  \parindent= paragraph indent
  {\bf}  % thm head font
  {. }    % punctuation after thm head
  { }    % space after thm head: `` ``=normal \newline=linebreak
  {}     % thm head specification
\theoremstyle{theorem}
\newtheorem{thm}{Theorem}[section]  % 1st argument is your name for it
\newtheorem{lemma}[thm]{Lemma}     % 2nd argument is what is printed

\newtheorem{prop}[thm]{Proposition}

\newtheoremstyle{definition}
  {12pt}		  % space above
  {0pt}  % space below
  {}  % bofy font
  {\parindent}     % ident - empty=no indent,  \parindent= paragraph indent
  {\bf}  % thm head font
  {. }    % punctuation after thm head
  { }    % space after thm head: `` ``=normal \newline=linebreak
  {}     % thm head specification
\theoremstyle{definition}

\newtheorem{ex}[thm]{Example}

\renewcommand{\proofname}{Proof}

\makeatletter
\renewenvironment{proof}[1][\proofname]{\par
  \pushQED{\qed}%
  \normalfont \partopsep=\z@skip \topsep=\z@skip
  \trivlist
  \item[\hskip\labelsep
        \scshape
    #1\@addpunct{.}]\ignorespaces
}{%
  \popQED\endtrivlist\@endpefalse
}
\makeatother

\makeatletter
\renewcommand*\@maketitle{%
  \normalfont\normalsize
  \@adminfootnotes
  \@mkboth{\@nx\shortauthors}{\@nx\shorttitle}%
  \global\topskip42\p@\relax % 5.5pc   "   "   "     "     "
  \@settitle
  \ifx\@empty\authors \else {\vskip 1em
\vtop{\centering\shortauthors\@@par}} \fi
  \ifx\@empty\@date \else {\vskip 1em \vtop{\centering\@date\@@par}}\fi % MY CHANGE
  \ifx\@empty\@dedicatory
  \else
    \baselineskip18\p@
    \vtop{\centering{\footnotesize\itshape\@dedicatory\@@par}%
      \global\dimen@i\prevdepth}\prevdepth\dimen@i
  \fi
  \@setabstract
  \normalsize
  \if@titlepage
    \newpage
  \else
    \dimen@34\p@ \advance\dimen@-\baselineskip
    \vskip\dimen@\relax
  \fi
} % end \@maketitle
\renewcommand*\@adminfootnotes{%
  \let\@makefnmark\relax  \let\@thefnmark\relax
  \ifx\@empty\@subjclass\else \@footnotetext{\@setsubjclass}\fi
  \ifx\@empty\@keywords\else \@footnotetext{\@setkeywords}\fi
  \ifx\@empty\thankses\else \@footnotetext{%
    \def\par{\let\par\@par}\@setthanks}%
  \fi
\thispagestyle{titlepage}
}
\makeatother

\begin{document}

\title[Balancing over the alternating group]{\large The balancing index over the alternating group}

\author{Peter J.~Dukes}

%\email{dukes@uvic.ca}

\author{Georgia Penner}

%\email{dukes@uvic.ca, georgiap@uvic.ca}
\address{
Department of Mathematics and Statistics,
University of Victoria, Victoria, Canada
}

\thanks{This research is supported by NSERC grant RGPIN-2024-03966}

\date{\today}

\begin{abstract}
The balancing index of a polynomial $f \in \Z[x_1,\dots,x_n]$ is the least positive sum of coefficients in an integer linear combination of permuted copies of $f$ which produces a symmetric polynomial.  Here we consider the restricted problem in which only even permutations are used.  
\end{abstract}

\maketitle
\hrule

\bigskip

\section{Introduction}
\label{sec-intro}

Let $\Z[S_n]$ act on the ring of polynomials $\Z[x_1,\dots,x_n]$ in the natural way by permuting variables.  In particular, for $\sigma \in S_n$ and $f \in \Z[x_1,\dots,x_n]$, 
\begin{equation}
\label{eq:permute}
\sigma(f)(x_1,\dots,x_n) = f(x_{\sigma(1)},\dots,x_{\sigma(n)}).
\end{equation}
We extend \eqref{eq:permute} by linearity for the action under the group ring. We say that $f$ is \emph{symmetric} if $\sigma(f)=f$ for all $\sigma \in S_n$. Define the \emph{balancing index} $\bal(f)$ of $f$ to be 
\begin{equation}
\label{eq:bal}
\bal(f) = \gcd\{\alpha(1): \alpha \in \Z[S_n] \text{ with } \alpha(f) \text{ symmetric}\}.
\end{equation}
That is, this is the least positive sum of coefficients in an integer linear combination of permuted copies of $f$ which lands in the ring of symmetric polynomials $\Z[x_1,\dots,x_n]^{S_n}$. Since $\sum_{\sigma \in S_n} \sigma(f)$ is symmetric for any $f$, it follows that \eqref{eq:bal} is well-defined with $\bal(f) \mid n!$.

The case where $f$ is a squarefree quadratic is related to a problem on edge-decompositions of graphs.  A graph $G$ on vertex set $[n]:=\{1,\dots,n\}$ can be associated with the quadratic
$$f_G=\sum_{ij \in E(G)} x_i x_j.$$
The balancing index of $f_G$ is considered in \cite{dVD}, where it is abbreviated simply $\bal(G)$.  This has combinatorial motivation \cite{CC,Wilson75} because knowing $\bal(G)$ informs which sufficiently large multiples of the complete graph 
$\lambda K_n$ admit an edge-decomposition into copies of $G$.  The papers \cite{dVD,WW} show, with two different methods, how to compute $\bal(G)$ as a function of the degree sequence of $G$.  More generally, square matrices were considered in \cite{dVD} as edge-weighted directed graphs. It was shown that the balancing index depends only on some `local' information, such as row and column sums, diagonal entries, and other data involving small alternating expressions over $S_3$ and the Klein 4-group.  Even though our variables $x_1,\dots,x_n$ are assumed to commute, we can model directions with different exponents in two-variable monomials.  In more detail, for a matrix $A \in \Z^{n \times n}$, it is possible to reformulate its balancing index in the polynomial setting as $\bal(A):=\bal(f_A)$, where
\begin{equation}
\label{eq:matrix-poly}
f_A = \sum_{i,j=1}^n A_{ij} x_i^2 x_j.
\end{equation}

Given a subgroup $\Gamma$ of $S_n$ and $f \in \Z[x_1,\dots,x_n]$, define $\bal(f,\Gamma)$ to equal the $\gcd$ of $\alpha(1)$ over all $\alpha \in \Z[\Gamma]$ with $\alpha(f)$ is symmetric.  When there are no such $\alpha$, we define $\bal(f,\Gamma) = 0$.  It is an easy observation that $\bal(f)$ divides $\bal(f,\Gamma)$.

Note that, for proper subgroups $\Gamma$, it could be the case that $\bal(f,\Gamma)$ depends on the ordering of the variables. In particular, for a graph $G$, if we let $\bal(G,\Gamma)$ denote the balancing index for the squarefree quadratic $f_G$, this quantity may depend on the labelling of the vertices of $G$.  However, since the alternating group $A_n$ is normal in $S_n$, the balancing index over $A_n$ is, in fact, invariant under permutation of variables/vertices.  To see this, suppose $\sum_{\sigma \in A_n} c_\sigma \sigma(f) = g$, with $g$ a symmetric polynomial.  Then for any $\alpha \in S_n$, we can re-index the sum using conjugation to get 
$\sum_{\sigma \in A_n} c_{\alpha^{-1} \sigma \alpha} \sigma(\alpha f) = \alpha g = g.$  This motivates studying the (integer) ratio
$\bal(f,A_n)/\bal(f)$, which we denote by $R(f)$.

Our main goal in this note is to obtain upper bounds on $R(f)$.
We give three results for special classes of polynomials.  The first two generalize the situation for hypergraphs and symmetric matrices, while the third generalizes that for arbitrary matrices.

\begin{thm}
\label{thm:main3}
Suppose $f \in \Z[x_1,\dots,x_n]$ is squarefree with $n \ge 3$.  Then $R(f) \mid 3$.
%\bal(f,A_n)$ equals either $\bal(f)$ or $3\, \bal(f)$.
\end{thm}

\begin{thm}
\label{thm:main1}
Suppose $f \in \Z[x_1,\dots,x_n]$ is squarefree with with $n \ge 2 \deg(f)+2$.
Then $R(f)=1$.
\end{thm}

\begin{thm}
\label{thm:main62}
Suppose $f \in \Z[x_1,\dots,x_n]$ has at most two variables in each monomial term.  Then 
$R(f)$ divides $6$ when $n=5$, and $R(f)$ divides $2$ when $n \ge 6$.
\end{thm}

The proofs of these results are all straightforward consequences of the existence of certain elements of $\Z[S_n]$ which we call `null designs'.  These are permutation analogues of similar objects for block designs.  In the next section, we set up the framework of null designs and show how it gives bounds on $R(f)$.  We first focus on squarefree polynomials and prove
Theorems~\ref{thm:main3} and \ref{thm:main1}.  More general polynomials are then briefly considered, including a proof of Theorem~\ref{thm:main62}.  In Section~\ref{sec:counterex}, we show that in many cases the bounds on $R(f)$ are tight by exhibiting counterexamples to weakening the hypotheses for smaller values of $n$.  We conclude with a discussion of some open problems and possible next directions.

\section{Upper bounds on $R(f)$}

\subsection{Null designs}

For a subset $A \subseteq [n]$, let $m_A:=\prod_{i \in A} x_i$, the squarefree monomial using variables indexed by $A$.
We say that an element $\nu \in \Z[S_n]$ is a \emph{null} $t$-\emph{design} if $\nu(m_A)=0$ for every $A \subseteq [n]$ with $|A|\le t$.  In particular, $\nu$ is a null $0$-design if and only if its sum of coefficients equals zero.  A related but slightly different notion of `permutation design' was introduced by Godsil in \cite{Godsil} and also studied more recently in
in \cite{SLS,Sole}.  

One potentially helpful way to think of a null $t$-design $\nu = \sum_{\sigma \in S_n} c_\sigma \sigma$
is as a $(\sum_\sigma |c_\sigma|) \times n$ array divided vertically into `positive' and `negative' sections, where $\sigma(1),\dots,\sigma(n)$ appears as a row exactly $|c_\sigma|$ times in the appropriate section. The $t$-design property is then checked as follows.  For any selection of $t$ or fewer columns, the number of times any subset appears as a positive row in those columns is the same as the number of times it appears as a negative row in those columns.

\begin{ex}
\label{ex:alt-null-design}
For any $n \ge 3$, $\perm{()}+\perm{(123)}+\perm{(132)}-\perm{(12)}-\perm{(13)}-\perm{(23)}$ is a null $3$-design in $\Z[S_3]$.  In fact, it is a null $n$-design in $S_n$ for any $n \ge 3$.  We give the associated array in the case $n=4$ for illustration in Table~\ref{tab:n4t4}.  In (say) columns $1$ and $3$, each of the sets $\{1,2\}, \{1,3\}, \{2,3\}$ occurs once in the upper half and once in the lower half.  A similar property holds for any selection of columns.
\end{ex}

\begin{table}[htbp]
\parbox{.49\linewidth}{
\centering
\captionsetup{width=.9\linewidth}
\begin{tabular}{lccccc}
% & {1} & {2} & {3} & {4} \\
\hline
$\perm{()}$ & even & {1} & {2} & {3} & {4} \\ 
$\perm{(123)}$ & even & {2} & {3} & {1} & {4} \\ 
$\perm{(132)}$ & even & {3} & {1} & {2} & {4} \\ 
\hline
$\perm{(12)}$ & odd  & {2} & {1} & {3} & {4} \\
$\perm{(13)}$ & odd & {3} & {2} & {1} & {4} \\ 
$\perm{(23)}$ & odd & {1} & {3} & {2} & {4} \\ 
\hline
\end{tabular}
\caption{A null $4$-design in $\Z[S_4]$.}
\label{tab:n4t4}
}
\hfill
\parbox{.49\linewidth}{
\centering
\captionsetup{width=.9\linewidth}
\begin{tabular}{lccccc}
\hline
$\perm{()}$ & even & {1} & {2} & {3} & {4} \\ 
$\perm{(123)}$ & even & {2} & {3} & {1} & {4} \\ 
$\perm{(1342)}$ & odd & {3} & {1} & {4} & {2} \\ 
\hline
$\perm{(12)}$ & odd & {2} & {1} & {3} & {4} \\
$\perm{(13)}$ & odd & {3} & {2} & {1} & {4} \\ 
$\perm{(234)}$ & even & {1} & {3} & {4} & {2} \\ 
\hline
\end{tabular}
\caption{A null $1$-design in $\Z[S_4]$.}
\label{tab:n4t1}
}
\end{table}

More generally, for $n \ge k \ge 3$, the group ring element $\sum_{\sigma \in S_k} \mathrm{sgn}(\sigma) \sigma$ is a null $n$-design.  Given any two subsets $A,B \subseteq [k]$ with $|A|=|B|$, let $\Pi_{A,B}=\{\sigma \in S_k : \sigma(A)=B\}$.  We claim that $\Pi_{A,B}$ contains equally many even and odd permutations.  Consider the case where $|[k] \cap A| \ge 2$. Without loss of generality, say $\{1,2\} \subseteq A$. Then a one-to-one correspondence between even and odd permutations in $\Pi_{A,B}$ arises from pre-multiplying by $\perm{(12)}$.  A similar method can be applied if $|[k] \setminus A| \ge 2$, choosing a transposition in $S_k$ disjoint from $A$.   Finally, one of these inequalities holds because $k \ge 3$.

We now return to considering the balancing index over the alternating group. 
The following result uses null $t$-designs to bound the quantity $R(f)=\bal(f,A_n)/\bal(f)$ when $f$ is squarefree of degree at most $t$.

\begin{prop}
\label{prop:sqfree}
Suppose $\nu$ is a null $t$-design in $\Z[S_n]$.  
Let $\nu_0$ be the projection of $\nu$ onto $\Z[A_n]$.
Then for every squarefree polynomial $f$ in at least $n$ variables with degree at most $t$, we have $R(f) \mid \nu_0(1)$.
\end{prop}

\begin{proof}
Suppose $\alpha \in \Z[S_n]$ is such that $\alpha(f)$ is symmetric.
Write $\alpha=\sum_{\sigma \in S_n} a_\sigma \sigma$.
Decompose $\alpha=\alpha_0+\alpha_1$ with $\alpha_0 \in \Z[A_n]$ and 
$\alpha_1 \in \Z[S_n \setminus A_n]$.  We similarly decompose $\nu=\nu_0+\nu_1$.
Consider the combination $\beta = \nu_0 \alpha_0 - \nu_1 \alpha_1 \in \Z[A_n]$.
We compute
\begin{align*}
\beta(f)
%=(\nu_0 \alpha_0 - \nu_1 \alpha_1)(f)
&=
\sum_{\sigma \in A_n} a_\sigma \nu_0 \sigma(f) -
\sum_{\sigma \in S_n \setminus A_n} a_\sigma \nu_1 \sigma(f)\\
&=
\sum_{\sigma \in A_n} a_\sigma \nu_0 \sigma(f) +
\sum_{\sigma \in S_n \setminus A_n} a_\sigma \nu_0 \sigma(f)\\
&=
\nu_0 \alpha(f) = \nu_0(1) \alpha(f),
\end{align*}
where the first equality uses linearity and the definition of $\beta$,
the second uses that $\nu_0(g)=-\nu_1(g)$ for polynomials $g$ of degree at most $t$, and the last equality follows from $\alpha(f)$
being symmetric.
\end{proof}

Theorem~\ref{thm:main3} now follows as a direct consequence of Example~\ref{ex:alt-null-design} and Proposition~\ref{prop:sqfree}.  That is, $\bal(f,A_n)$ equals either $\bal(f)$ or $3\, \bal(f)$ for any squarefree polynomial in at least three variables.

\begin{ex}
\label{ex:null-1-design}
A null $1$-design is given by $\nu=\perm{()}+\perm{(123)}+\perm{(1342)}-\perm{(12)}-\perm{(13)}-\perm{(234)}$ and displayed as an array in Table~\ref{tab:n4t1}.  For each column, the upper and lower sections agree (as multisets).  Compared with the null design of Example~\ref{ex:alt-null-design}, this one is noteworthy for having mixed signs on even and odd permutations, resulting in $\nu_0(1)=1$.
\end{ex}

Suppose $f \in \Z[x_1,\dots,x_n]$ is a linear polynomial or, more generally, has at most one variable per term.  It is easy to see that $R(f)=1$ when $n=3$.  The null $1$-design of Example~\ref{ex:null-1-design} shows that this also holds for all $n \ge 4$.
In what follows, we use various null (ordered) designs to obtain upper bounds on $R(f)$ in other cases.

\subsection{Squarefree polynomials}

As already shown, Theorem~\ref{thm:main3} is an easy consequence of the null design $\perm{()}+\perm{(123)}+\perm{(132)}-\perm{(12)}-\perm{(13)}-\perm{(23)}$.  Theorem~\ref{thm:main1} improves the conclusion to $R(f)=1$ under the extra assumption $n \ge 2t+2$ when $\deg(f) \le t$.  To prove this, a result on products of null designs is helpful.  

For $\alpha =\sum_\sigma c_\sigma \sigma \in \Z[S_n]$, define its \emph{support} as
$\mathrm{supp}(\alpha) = \bigcup \{\{x:\sigma(x) \neq x\} : \sigma \in S_n \text{ with } c_\sigma \neq 0\}$.
If $\nu$ is a null $t$-design in $\Z[S_n]$, then $\nu (m_A) = 0$ for any $A \subseteq [n]$ with $|A \cap \mathrm{supp}(\nu)| \le t$.  We have the following observation on products of null designs with disjoint supports.

\begin{lemma}
\label{lem:product}
Suppose $\mu$ is a null $s$-design in $\Z[S_n]$, $\nu$ is a null $t$-design in $\Z[S_n]$, and $\mu$ and $\nu$ have disjoint supports.  Then $\mu \nu$ is a null $(s+t+1)$-design in $\Z[S_n]$.
\end{lemma}

\begin{proof}
Consider any $(s+t+1)$-subset $A \subseteq [n]$.  Then either $|A \cap \mathrm{supp}(\mu)| \le s$ or $|A \cap \mathrm{supp}(\nu)| \le t$.  It follows that $\mu \nu(m_A) = \nu \mu(m_A) = 0$, and hence $\mu \nu$ is a null $(s+t+1)$-design in $\Z[S_n]$.
\end{proof}

Let $\mu=\perm{()}-\perm{(12)}$.  By a minor abuse of notation, let $\mu^{t+1}$ denote a product of $t+1$ copies of $\mu$ with mutually disjoint supports in $\Z[S_n]$, $n \ge 2t+2$. Since $\mu$ is a null $0$-design, applying Lemma~\ref{lem:product} $t$ times implies that $\mu^{t+1}$ is a null $t$-design in $\Z[S_n]$.  This is closely related to the notion of a \textit{pod} in \cite{GJ}.  An example for $t=2$ is
$$\mu^3=\perm{()}-\perm{(12)}-\perm{(34)}-\perm{(56)}+\perm{(12)(34)}+\perm{(12)(56)}+\perm{(34)(56)}-\perm{(12)(34)(56)}.$$
Everything is now in place to prove Theorem~\ref{thm:main1}.

\begin{proof}[Proof of Theorem~\ref{thm:main1}]
Suppose $f$ is squarefree with $\deg(f) \le t$.  For $n \ge 2t+2$, $\mu^{t+1}$ as defined above is a null $t$-design in $\Z[S_n]$. 
We have $(\mu^{t+1})_0 = \sum_{i \ge 0} \binom{t+1}{2i} = 2^t$.
By Proposition~\ref{prop:sqfree}, $R(f) \mid 2^t$.  But from Theorem~\ref{thm:main3} we also have $R(f) \mid 3$.  It follows that $R(f)=1$.
\end{proof}

For the sake of interest, we exhibit in Table~\ref{tab:n6t2} a certain null $2$-design in $\Z[S_6]$ found with the help of a computer.  Note that the net coefficient total on even permutations is $4-3=1$.  This gives a slightly more direct verification that $R(f)=1$ for squarefree quadratics $f$ in at least $6$ variables.  The quadratic case applies to graphs of order $n$, and, more generally, to symmetric $n \times n$ matrices.  Informally, these objects can all be balanced using only the alternating group for $n \ge 6$.

\begin{table}
\begin{tabular}{lccccccc}
\hline
$\perm{()}$ & even & {1} & {2} & {3} & {4} & {5} & {6} \\
$\perm{(1234)}$ & odd & {2} & {3} & {4} & {1} & {5} & {6} \\
$\perm{(12)(345)}$ & odd & {2} & {1} & {4} & {5} & {3} & {6} \\
$\perm{(123)(45)}$ & odd & {2} & {3} & {1} & {5} & {4} & {6} \\
$\perm{(12435)}$ & even  & {2} & {4} & {5} & {3} & {1} & {6} \\
$\perm{(34)(56)}$ & even  & {1} & {2} & {4} & {3} & {6} & {5} \\
$\perm{(12)(56)}$ & even & {2} & {1} & {3} & {4} & {6} & {5} \\
\hline
$\perm{(123)}$ & even & {2} & {3} & {1} & {4} & {5} & {6} \\
$\perm{(12345)}$ & even & {2} & {3} & {4} & {5} & {1} & {6} \\
$\perm{(124)(35)}$ & odd & {2} & {4} & {5} & {1} & {3} & {6} \\
$\perm{(34)}$ & odd & {1} & {2} & {4} & {3} & {5} & {6} \\
$\perm{(12)(45)}$ & even & {2} & {1} & {3} & {5} & {4} & {6} \\
$\perm{(56)}$ & odd & {1} & {2} & {3} & {4} & {6} & {5} \\
$\perm{(12)(34)(56)}$ & odd & {2} & {1} & {4} & {3} & {6} & {5} \\
\hline
\end{tabular}
\caption{A null $2$-design in $\Z[S_6]$.} %; coefficients are $+1/-1$ in upper/lower sections.}
\label{tab:n6t2}
\end{table}

\subsection{General polynomials}

To handle more general polynomials, it is useful to have an ordered variant of null designs.  
For $A \subseteq [n]$, let $m_A^*:=\prod_{i \in A} x_i^i$. The specific values of the exponents are not important, so long as they are distinct on each variable.  (This is essentially the same as taking a product of distinct non-commuting variables with indices in $A$.)
We shall call $\nu \in \Z[S_n]$ a \emph{null ordered} $t$-\emph{design} if $\nu(m_A^*)=0$ for every $A \subseteq [n]$ with $|A|\le t$.  In the array representation, we now need the stronger property that in any selection of $t$ or fewer columns, the number of times any word (that is, a partial permutation) appears as a positive row in those columns is the same as the number of times it appears as a negative row in those columns.

We state below an ordered version of Proposition~\ref{prop:sqfree}.  The proof is nearly identical, with the key point again being that $\nu_0(g)=-\nu_1(g)$ when $g$ is a polynomial having at most $t$ variables per term.

\begin{prop}
\label{prop:ordered} 
Suppose $\nu$ is a null ordered $t$-design in $\Z[S_n]$.  
Let $\nu_0$ be the projection of $\nu$ onto $\Z[A_n]$.
Then for every polynomial $f \in \Z[x_1,\dots,x_n]$ with at most $t$ variables in every term, we have $R(f) \mid \nu_0(1)$.
\end{prop}

Lemma~\ref{lem:product} on products still applies in the ordered setting.  Any null $0$-design or $1$-design is trivially ordered, so these can be used as inputs.  In particular, if $\nu$ is a null $1$-design in $\Z[S_{n-2}]$, and $\tau$ denotes the transposition $(n-1,n)$, then $\nu^*:=\nu(1-\tau)$ is a null ordered $2$-design in $\Z[S_n]$.  We have 
\begin{equation}
\label{eq:doubling-nu0}
\nu^*_0(1)=\nu_0(1)-\nu_1(1)=2\nu_0(1).
\end{equation}

The following examples give two special cases.

\begin{ex}
The null $3$-design $\nu=\perm{()}+\perm{(123)}+\perm{(132)}-\perm{(12)}-\perm{(13)}-\perm{(23)}$ of Example~\ref{ex:alt-null-design} is also a null ordered $1$-design in $\Z[S_3]$.  With $\tau=\perm{(45)}$, we obtain a null ordered $2$-design $\nu^*=\nu(1-\tau)$ in $\Z[S_5]$.  From \eqref{eq:doubling-nu0}, this has $\nu^*_0(1)=6$.
\end{ex}

\begin{ex}
If we take the null (ordered) $1$-design $\nu$ of Example~\ref{ex:null-1-design} and $\tau=\perm{(56)}$, we obtain a null ordered $2$-design $\nu^*=\nu(1-\tau)$ in $\Z[S_6]$.  From \eqref{eq:doubling-nu0}, this has $\nu^*_0(1)=2$.
\end{ex}

The proof of Theorem~\ref{thm:main62} now follows directly from Proposition~\ref{prop:ordered} and these two examples.  More generally, using $\mu^{t+1}$ from the squarefree case, we can conclude $R(f) \mid 2^t$ for polynomials $f$ of degree $t$, provided $n \ge 2t+2$.  The power of $2$ can be lowered somewhat by using $\left \lfloor \frac{t-1}{2} \right\rfloor$ disjoint copies of $\nu$ from Example~\ref{ex:null-1-design}.  We are unsure whether such a bound is close to the truth for general $t$.  Some examples are given in the next section that push against this bound for $t=2$.

\section{Counterexamples for small $n$}
\label{sec:counterex}

We next give counterexamples to weakening the bound $n \ge 2 \deg(f)+2$ of Theorem~\ref{thm:main1}.  That is, the extra factor of three in the alternating balancing index is necessary for small $n$, at least in certain cases.  We first consider squarefree polynomials through graphs and hypergraphs.

\begin{ex}
Let $G=C_5$, the cycle on 5 vertices.  It is clear that $\bal(G)=2$; if the cycle follows the vertices in the natural order $1,2,3,4,5$, then applying $\perm{()}+\perm{(2354)}$ produces $K_5$ as a sum of two complementary $5$-cycles. 
We also have $\bal(G,A_5) \mid 6$ via 
\begin{equation}
\label{eq:balC5}
\perm{()}+\perm{(234)}+\perm{(243)}+\perm{(345)}+\perm{(354)}+\perm{(23)(45)} \in \Z[A_5].
\end{equation}
We can show that $\bal(G, A_5) \neq 2$ as follows.  With the same vertex labelling as above, call an edge of $K_5$ `short' or `long' according to whether it is in $G$ or $\overline{G}$, respectively.  For each $\sigma \in A_5$, it turns out that $\sigma(G)$ has exactly zero or three long edges.  Suppose $\alpha \in \Z[A_5]$ satisfies $\alpha(G)=\lambda K_5$ for some positive integer $\lambda$.  Since $K_5$ has five long edges and $\gcd(3,5)=1$, it follows that $3\mid \lambda$ and hence $3 \mid \bal(G)$.
\end{ex}

For $n=4$, an example with similar properties is $G=P_4$, the path on $4$ vertices.  It is self-complementary, so $\bal(P_4)=2$.  However, it can be checked that $\bal(P_4,A_4)=6$.

Here is an example for the case of squarefree cubics.

\begin{ex}
Let $H$ be the hypergraph on $6$ vertices with edge set 
$$\{\{1, 2, 6\}, \{2, 3, 6\}, \{3, 4, 6\}, \{4, 5, 6\}, \{1, 5, 6\}, \{1, 2, 4\}, \{2, 3, 5\}, \{1, 3, 4\}, \{2, 4, 5\}, \{1, 3, 5\}\}.$$
This forms the block set of a $(6,3,2)$-BIBD, and is invariant under the dihedral group generated by the rotation $\perm{(12345)}$ and reflection $\perm{(25)(34)}$.  It is easy to see that $\bal(H)=2$ via $\perm{()} + \perm{(14)(25)(36)}$.  
We can show similarly to the previous example that $\bal(H, A_6)=6$. The same group ring element \eqref{eq:balC5} shows $\bal(H,A_6) \mid 6$.   To show equality, we first partition the possible 3-subsets of $\{1,\dots,6\}$ into four types according to their triangular shapes.  Figure~\ref{fig:triangle-types} shows these shapes; the middle vertex is $6$ and the outer pentagon is has labels $1,\dots,5$, in order.
\begin{figure}[htbp]
\begin{tikzpicture}
\begin{scope}
	\node [style=none] at (0,-1.3) {A};
	% vertices
	\node [style=vertex] (0) at (0,0) {};
	\node [style=vertex] (1) at (90:1) {};
	\node [style=vertex] (2) at (162:1) {};
	\node [style=vertex] (3) at (234:1) {};
	\node [style=vertex] (4) at (306:1) {};
	\node [style=vertex] (5) at (18:1) {};
	% edges
	\draw [style=edge] (1) to (5);
	\draw [style=edge] (5) to (0);
	\draw [style=edge] (0) to (1);
\end{scope}
\begin{scope}[xshift=4cm]
	\node [style=none] at (0,-1.3) {B};
	% vertices
	\node [style=vertex] (0) at (0,0) {};
	\node [style=vertex] (1) at (90:1) {};
	\node [style=vertex] (2) at (162:1) {};
	\node [style=vertex] (3) at (234:1) {};
	\node [style=vertex] (4) at (306:1) {};
	\node [style=vertex] (5) at (18:1) {};
	% edges
	\draw [style=edge] (1) to (4);
	\draw [style=edge] (4) to (0);
	\draw [style=edge] (0) to (1);
\end{scope}
\begin{scope}[xshift=8cm]
	\node [style=none] at (0,-1.3) {C};
	% vertices
	\node [style=vertex] (0) at (0,0) {};
	\node [style=vertex] (1) at (90:1) {};
	\node [style=vertex] (2) at (162:1) {};
	\node [style=vertex] (3) at (234:1) {};
	\node [style=vertex] (4) at (306:1) {};
	\node [style=vertex] (5) at (18:1) {};
	% edges
	\draw [style=edge] (1) to (5);
	\draw [style=edge] (5) to (3);
	\draw [style=edge] (3) to (1);
\end{scope}
\begin{scope}[xshift=12cm]
	\node [style=none] at (0,-1.3) {D};
	% vertices
	\node [style=vertex] (0) at (0,0) {};
	\node [style=vertex] (1) at (90:1) {};
	\node [style=vertex] (2) at (162:1) {};
	\node [style=vertex] (3) at (234:1) {};
	\node [style=vertex] (4) at (306:1) {};
	\node [style=vertex] (5) at (18:1) {};
	% edges
	\draw [style=edge] (1) to (5);
	\draw [style=edge] (5) to (4);
	\draw [style=edge] (4) to (1);
\end{scope}
\end{tikzpicture}
\caption{Examples of the four triangular shapes.}
\label{fig:triangle-types}
\end{figure}
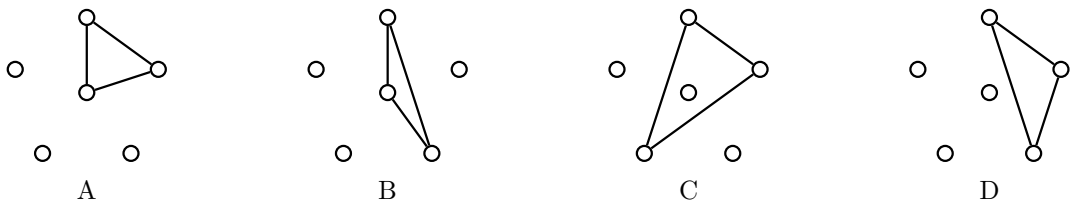

There are five 3-subsets of each type. So, in order to balance $H$, we need $5k$ triangles of each type for some $k \in \mathbb{N}$. A simple check shows that there are two possibilities for $\sigma \in A_6$: either $\sigma(H)$ has five $3$-subsets each of types A and C, or it has two each of types A and C and three each of types B and D.  An integral linear combination of these has a multiple of three triples of types B and D.  Since $\gcd(3,5)=1$, we obtain $3 \mid \bal(H,A_5)$ similar to the 
previous example.%  Therefore, $\bal(H,A_6)=6$. 
\end{ex}

We also found a different $3$-uniform hypergraph $H'$ on $6$ vertices and $12$ edges having $\bal(H')=5$ and $\bal(H',A_6)=15$.

Moving on to the ordered (not necessarily squarefree) setting, we note that the lower bounds on $n$ are necessary in Theorem~\ref{thm:main62}.  For convenience, we display the following examples using digraphs and matrices.

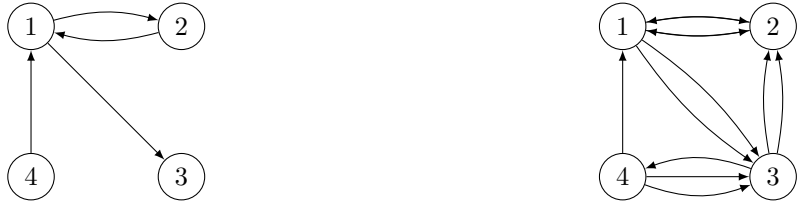
\begin{figure}[htbp]
\begin{tikzpicture}
\tikzset{vertex/.style = {shape=circle,draw,minimum size=1.5em}}
\tikzset{edge/.style = {->,> = latex}}
% vertices
\node[vertex] (a) at  (0,2) {$1$};
\node[vertex] (b) at  (2,2) {$2$};
\node[vertex] (c) at  (2,0) {$3$};
\node[vertex] (d) at  (0,0) {$4$};
%edges
\draw[edge] (a) to[bend left=15] (b);
\draw[edge] (b) to[bend left=15] (a);
\draw[edge] (d) to (a);
\draw[edge] (a) to (c);
\end{tikzpicture}
\hspace{5cm}
\begin{tikzpicture}
\tikzset{vertex/.style = {shape=circle,draw,minimum size=1.5em}}
\tikzset{edge/.style = {->,> = latex}}
% vertices
\node[vertex] (a) at  (0,2) {$1$};
\node[vertex] (b) at  (2,2) {$2$};
\node[vertex] (c) at  (2,0) {$3$};
\node[vertex] (d) at  (0,0) {$4$};
%edges
\draw[edge] (a) to[bend left=10] (b);
\draw[edge] (b) to[bend left=10] (a);
\draw[edge] (a) to[bend right=10] (b);
\draw[edge] (b) to[bend right=10] (a);
\draw[edge] (d) to (a);
\draw[edge] (a) to[bend left=10] (c);
\draw[edge] (a) to[bend right=10] (c);
\draw[edge] (c) to[bend left=10] (b);
\draw[edge] (c) to[bend right=10] (b);
\draw[edge] (d) to (c);
\draw[edge] (c) to[bend right=20] (d);
\draw[edge] (d) to[bend right=20] (c);
\end{tikzpicture}
\caption{Counterexamples to Theorem~\ref{thm:main62} for $n=4$.}
\label{fig:digraphs}
\end{figure}

\begin{ex}
\label{ex:n4-a12-b3}
The directed graph $D$ on the left in Figure~\ref{fig:digraphs} has balancing index $3$ over $S_4$ and balancing index $12$ over $A_4$.  The balancing over $S_4$ can be done via
$$\perm{()}-\perm{(12)}+\perm{(123)}+\perm{(1234)}+\perm{(1342)}-\perm{(134)}+\perm{(14)(23)}.$$
We claim that the $12$ copies of $D$ under $A_4$ are linearly independent over $\Q$. This can be established by applying
$$\perm{()}-\perm{(123)}+\perm{(124)}-\perm{(142)}-\perm{(143)}-\perm{(234)}+\perm{(243)}+\perm{(12)(34)}+\perm{(14)(23)}$$
to $D$, producing a single edge of nonzero weight.  Using $A_4$ we can therefore span the space of all directed weighted graphs on $4$ vertices, which has dimension $12$. It follows that the unique linear combination of the $12$ copies of $D$ under $A_4$ having all equal edge multiplicities is (up to scaling) the one which uses all copies with equal coefficients.  Thus $\bal(D,A_4)=12$.
\end{ex}

\begin{ex}
\label{ex:n4-a12-b1}
The directed multigraph $D'$ on the right in Figure~\ref{fig:digraphs} has balancing index $1$ over $S_4$ and balancing index $12$ over $A_4$. The balancing over $S_4$ can be done via
$$2\perm{()}-3\perm{(12)}-2\perm{(23)}+\perm{(14)}+3\perm{(123)}+3\perm{(1342)}-2\perm{(1423)}+\perm{(1234)}.$$
Verification that $\bal(D',A_4)=12$ can be carried out as in the previous example.
\end{ex}

The factor of two in Theorem~\ref{thm:main62} is still necessary for $n=5$ and $n=6$, as shown in the following examples.  To compute the balancing index values, we used a Smith normal form calculation as described in \cite{dVD}.  

\begin{ex}
\label{ex:n5-a6-b1}
The matrix
$$A=\left[\begin{array}{rrrrr}
 4 & 0 & 1 & 4 & 1 \\
 3 & 8 & 9 & 5 & 3 \\
 4 & 1 & 8 & 5 & 7 \\
 4 & 6 & 0 & 9 & 6 \\
 9 & 2 & 8 & 2 & 1
 \end{array}\right]$$
has balancing index $1$ over $S_5$ and balancing index $6$ over $A_5$.
\end{ex}

\begin{ex}
\label{ex:n6-a60-b30}
The matrix
$$A=\left[\begin{array}{rrrrrr}
4 & 0 & 4 & 0 & 5 & 5\\
0 & 0 & 0 & 2 & 5 & 5\\
4 & 3 & 0 & 3 & 4 & 1\\
0 & 3 & 0 & 0 & 0 & 5\\
5 & 4 & 5 & 5 & 4 & 2\\
2 & 1 & 3 & 1 & 4 & 3
 \end{array}\right]$$
has balancing index $30$ over $S_6$ and balancing index $60$ over $A_6$.
\end{ex}

\section{Conclusion}

For squarefree $f \in \Z[x_1,\dots,x_n]$ of degree $t$, we have seen that $\bal(f,A_n)$ equals either $\bal(f)$ or $3\, \bal(f)$, with the former possibility prevailing for all $f$ when $n \ge 2t+2$.  We have examples where the multiple of $3$ is necessary for $(n,t)=(2,5)$ and $(3,6)$.  We do not know how good the lower bound on $n$ is for larger $t$.  This would be a good question for further study.  
%Focusing on the case $t=2$, a likely easier problem would be classifying the symmetric $4 \times 4$ and $5 \times 5$ integer matrices for which the multiple of $3$ is required.  

If we drop the squarefree assumption, there remain many questions.  For a given degree bound $t$, we have seen that $R(f)$ is at most exponential in $t$ for $n \ge 2t+2$.  We are interested in whether $R(f)$ has a bound independent of $t$ for large $n$, and also whether $R(f)$ can get very large when $1<n/t<2$.  For the latter question, we found a computer-generated polynomial $f$ in five variables, all monomial terms having shape $(3,2,1)$, with $\bal(f)=1$ and $\bal(f,A_n)=60$.  One might think of this as a `rank 3' analog of the directed multigraph in Example~\ref{ex:n4-a12-b1}.

It should be noted that $\bal(f,A_n)$ may equal zero.  Indeed, a monomial such as $f=x_1 x_2^2 \cdots x_n^n$ with distinct exponents cannot be balanced without using all permutations in $S_n$. We do not have a classification of such polynomials.
More generally, given a pair $(n,t)$, one might ask which pairs of integers $(a,b)$ can be achieved as $(\bal(f,A_n),\bal(f))$ for some polynomial $f$ of degree $t$ in $n$ variables.  

Consider $n \times n$ matrices, which from \eqref{eq:matrix-poly} can be modeled as polynomials having one or two variables per term.
Example~\ref{ex:n6-a60-b30} shows that $\bal(f,A_n)$ could equal $2\, \bal(f)$ for $n \le 6$.  We do not know whether this factor of $2$ is needed for larger $n$.

It may be helpful for experimentation to implement a fast algorithm for computing the balancing index of larger matrices over the alternating group, as was done in \cite{dVD} for the symmetric group.  We briefly sketch one potential approach that may work for symmetric matrices.
Let
$$\gamma=\perm{()}+\perm{(23)(45)}-\perm{(234)}-\perm{(243)}+\perm{(24)(35)}-\perm{(12)(45)}+\perm{(12)(34)}-\perm{(12)(35)} \in \Z[A_5].$$
This $\gamma$ has the property that $\gamma(x_1)=x_1-x_2$, $\gamma(x_2)=x_2-x_1$, and $\gamma(x_i)=0$ for $i=3,4,5$.
In other words, $\gamma$ applies the same action as $\perm{()}-\perm{(12)}$ on single variables, and it uses only even permutations.  Taking the product of two copies of $\gamma$ on disjoint sets can then produce the same action on pairs as $\perm{()}-\perm{(12)}-\perm{(34)}+\perm{(12)(34)}$.  This latter combination was central to the algorithm in \cite{dVD}.  In the language of graph theory, this gadget allows for local adjustments on $4$-cycles to reduce the overall variance of edge weights.  Using $\gamma$, it appears that this same algorithm can be carried out using only permutations in $A_n$.

\section*{Acknowledgments}
The authors are grateful to Coen del Valle and Amarpreet Rattan for helpful discussions in early stages of this work.

\end{document}